\documentclass[onefignum,onetabnum]{siamart190516}

\usepackage{lipsum}
\usepackage{amsfonts}
\usepackage{graphicx}
\usepackage{enumerate}
\usepackage{epstopdf}
\usepackage{microtype}
\usepackage{algorithmic}

\newsiamremark{remark}{Remark}
\newsiamremark{hypothesis}{Hypothesis}
\crefname{hypothesis}{Hypothesis}{Hypotheses}
\newsiamthm{conjecture}{Conjecture}

\title{On the Laplacian spread of digraphs\thanks{This work was supported by the AMS-Simons Travel Grants, which are administered by the American Mathematical Society with support from the Simons Foundation.}}

\author{Wayne Barrett\thanks{Department of Mathematics, Brigham Young University, Provo, UT (\email{wb@mathematics.byu.edu}, \email{ejevans@math.byu.edu}, \email{mkempton@mathematics.byu.edu}).}
\and Thomas R. Cameron\thanks{Department of Mathematics, Penn State Behrend, Erie, PA 
  (\email{trc5475@psu.edu}).}
\and Emily Evans\footnotemark[2]
\and H.~Tracy Hall\thanks{Hall Labs LLC, Provo, UT 
  (\email{H.Tracy@gmail.com}).}
\and Mark Kempton\footnotemark[2]}

\usepackage{amsopn}
\makeatletter
\newcommand{\oset}[3][0ex]{%
 \mathbin{\mathop{#3}\limits^{%
  \vbox to#1{\kern-0.7\ex@
   \hbox{$\scriptstyle#2$}\vss}}}}
\makeatother
\newcommand{\djoin}{\oset{\rightarrow}{\vee}}

\usepackage{mathtools}
\DeclarePairedDelimiter{\floor}{\lfloor}{\rfloor}
\newcommand*\comp[1]{\overline{#1}}
\newcommand{\iu}{{i\mkern1mu}}
\newcommand\abs[1]{\left|#1\right|}
\newcommand\norm[1]{\left\Vert#1\right\Vert}
\newcommand\re[1]{\operatorname{Re}\mathclose{\left(#1\right)}}

\newcommand\conv[1]{\operatorname{conv}\mathclose{\left(#1\right)}}

\renewcommand\sp[1]{\operatorname{sp}\mathclose{\left(#1\right)}}

\begin{document}

\maketitle

\begin{abstract}
In this article, we extend the notion of the Laplacian spread to simple directed graphs (digraphs) using the restricted numerical range. 
First, we provide Laplacian spread values for several families of digraphs.
Then, we prove sharp upper bounds on the Laplacian spread for all polygonal and balanced digraphs. 
In particular, we show that the validity of the Laplacian spread bound for balanced digraphs is equivalent to the Laplacian spread conjecture for simple undirected graphs, which was conjectured in 2011 and proven in 2021. 
Moreover, we prove an equivalent statement for weighted balanced digraphs with weights between $0$ and $1$. 
Finally, we state several open conjectures that are motivated by empirical data. 
\end{abstract}
\begin{keywords}
numerical range; directed graph; Laplacian matrix; Laplacian spread; algebraic connectivity
\end{keywords}
\begin{AMS}
05C20, 05C50, 15A18, 15A60, 52B20
\end{AMS}

\section{Introduction}\label{sec:intro}
Let $G$ be an undirected and unweighted simple graph (no loops nor multi-edges) of order $n$.
Also, let $L$ be the Laplacian matrix of $G$ and denote its eigenvalues by
\[
0=\lambda_{1}(G)\leq\lambda_{2}(G)\leq\cdots\leq\lambda_{n}(G).
\]
In~\cite{Fiedler1973}, Fiedler defined the \emph{algebraic connectivity} of the graph $G$ by $\alpha(G) = \lambda_{2}(G)$.
A related and useful quantity is $\beta(G) = \lambda_{n}(G) = n - \alpha(\comp{G})$, where $\comp{G}$ denotes the complement of the graph $G$. 
The \emph{Laplacian spread} of the graph $G$ is defined by $\sp{G} = \beta(G)-\alpha(G)$.

The Laplacian spread of a graph has received significant attention in the literature (for example, see~\cite{Afshari2019,Afshari2018,Andrade2016,Bao2009,Chen2009,Einollahzadeh2021,Fan2008,Liu2010,You2012,Zhai2011} and the references therein). 
Most notably, for our purposes, in~\cite{You2012,Zhai2011}, it was conjectured that the Laplacian spread satisfies
\begin{equation}\label{eq:lap-spread-graphs}
\sp{G}\leq n-1,
\end{equation}
for all graphs $G$.
Note that, since $\alpha(\comp{G})=n-\lambda_{n}(G)$, it follows that~\eqref{eq:lap-spread-graphs} can be re-written as
\begin{equation}\label{eq:lap-spread-graphs2}
\alpha(G) + \alpha(\comp{G}) \geq 1.
\end{equation}
In~\cite{Afshari2019}, it is shown that~\eqref{eq:lap-spread-graphs2} holding for all graphs $G$ of order $n\geq 2$ is equivalent to the following statement:
For any two orthonormal vectors $\mathbf{x},\mathbf{y}\in\mathbb{R}^{n}$ with zero mean and $n\geq 2$, 
\begin{equation}\label{eq:lap-spread-graphs3}
\norm{\nabla_{\mathbf{x}}-\nabla_{\mathbf{y}}} \geq 2,
\end{equation}
where $\nabla_{\mathbf{x}}\in\mathbb{R}^{\binom{n}{2}}$ is defined as the vector whose $ij$ entry is equal to $\abs{x_{i}-x_{j}}$, for all $i<j$.
More recently, in~\cite[Theorem 1]{Einollahzadeh2021}, it was shown that~\eqref{eq:lap-spread-graphs2} holds for all graphs $G$.
Hence,~\eqref{eq:lap-spread-graphs3} holds for all orthonormal vectors $\mathbf{x},\mathbf{y}\in\mathbb{R}^{n}$ with zero mean and $n\geq 2$. 

In this article, we extend the notion of the Laplacian spread to digraphs using the restricted numerical range as defined in~\cite{Cameron2022_RNR,Cameron2021_RNR}.
Note that the restricted numerical range is a closed convex set in the complex plane, and the Laplacian spread of a digraph can be viewed geometrically as the length of the real part of its restricted numerical range.
We use the digraph characterizations via the restricted numerical range in~\cite{Cameron2021_RNR} to give Laplacian spread values for several families of digraphs, and we provide sharp bounds on the Laplacian spread for all polygonal digraphs as defined in~\cite{Cameron2022_RNR}.
Moreover, we show that the Laplacian spread of balanced digraphs of order $n$ is bounded above by $(n-1)$.
In particular, we show that this statement is equivalent to the statements in~\eqref{eq:lap-spread-graphs2} and~\eqref{eq:lap-spread-graphs3}.
Then, we prove an equivalent statement for balanced digraphs with weights in the interval $[0,1]$.
Finally, we provide empirical evidence that the strengthened bound on the Laplacian spread conjectured in~\cite{Barrett2022} also holds for balanced digraphs, and we include several open conjectures regarding the families of digraphs that attain certain spread values.
\section{The Restricted Numerical Range}\label{sec:rnr}
Let $\Gamma = (V,E)$ denote an unweighted simple digraph of order $n$, where $V$ is the vertex set and $E\subseteq V\times V$ is the edge set. 
We denote the \emph{out-degree} of vertex $i\in V$ by $d^{+}(i)$, which is equal to the number of edges of the form $(i,j)\in E$.
Similarly, we denote the \emph{in-degree} of vertex $i\in V$ by $d^{-}(i)$, which is equal to the number of edges of the form $(j,i)\in E$. 
After indexing the vertex set as $V=\left\{1,2,\ldots,n\right\}$, we define the \emph{adjacency matrix} of $\Gamma$ by $A=[a_{ij}]_{i,j=1}^{n}$, where $a_{ij}=1$ if $(i,j)\in E$ and $a_{ij}=0$ otherwise. 
Moreover, we define the \emph{Laplacian matrix} of $\Gamma$ by $L=D-A$, where $A$ is the adjacency matrix of $\Gamma$ and $D$ is a diagonal matrix whose $i$th diagonal entry is $d^{+}(i)$.
We use functional notation to indicate the particular digraph when it is unclear from context; for example, $L(\Gamma)$ denotes the Laplacian matrix of the digraph $\Gamma$.

In general, the \emph{numerical range} (or \emph{field of values}) of a complex matrix $A\in\mathbb{C}^{n\times n}$ is defined as follows~\cite{Kippenhahn1951,Zachlin2008}:
\[
W(A) = \left\{\mathbf{x}^{*}A\mathbf{x}\colon~\mathbf{x}\in\mathbb{C}^{n},~\norm{\mathbf{x}}=1\right\},
\]
where $\norm{\cdot}$ denotes the Euclidean norm on complex vectors.
Since $L\mathbf{e}=0$ for any Laplacian matrix, where $\mathbf{e}$ is the all ones vector, we are interested in the \emph{restricted numerical range} of the Laplacian matrix, which is defined as follows~\cite{Cameron2021_RNR}:
\[
W_{r}(L) = \left\{ \mathbf{x}^{*}L\mathbf{x}\colon~\mathbf{x}\in\mathbb{C}^{n},~\mathbf{x}\perp\mathbf{e},~\norm{\mathbf{x}}=1\right\}.
\]
Clearly, $W_{r}(L)=\emptyset$ when $n=1$.
When convenient, we may refer to $W_{r}(L)$ as the restricted numerical range of a digraph, and mix the notation $W_{r}(L)$ with $W_{r}(\Gamma)$.

The definition of the restricted numerical range is motivated by its close connection to the algebraic connectivity for digraphs, which is defined as follows~\cite{Wu2005}:
The \emph{algebraic connectivity} of $\Gamma$ is given by
\[
\alpha(\Gamma) = \min_{\substack{\mathbf{x}\perp\mathbf{e} \\ \norm{\mathbf{x}}=1}}\mathbf{x}^{T}L\mathbf{x}.
\]
Another related and useful quantity is
\[
\beta(\Gamma) = \max_{\substack{\mathbf{x}\perp\mathbf{e} \\ \norm{\mathbf{x}}=1}}\mathbf{x}^{T}L\mathbf{x}.
\]
The proposition below summarizes this connection and the basic properties of the restricted numerical range~\cite{Cameron2022_RNR,Cameron2021_RNR}.
Note that we define a \emph{restrictor matrix} of order $n$ as an $n\times(n-1)$ matrix whose columns form an orthonormal basis for $\mathbf{e}^{\perp}$.
Given a restrictor matrix $Q$ of order $n$, we reference the matrix $Q^{*}LQ$ as a \emph{restricted Laplacian}.
Also, we use $\comp{\Gamma}$ to denote the \emph{complement digraph} of $\Gamma$, that is, the digraph whose edge set consists exactly of those directed edges not in $\Gamma$.
Finally, since the eigenvectors of $L$ corresponding to the zero eigenvalue are linearly independent~\cite[Theorem 2.4]{Cameron2020}, we are justified in referencing the zero eigenvalue corresponding to the eigenvector $\mathbf{e}$.
\begin{proposition}\label{prop:basic-rnr}
Let $\Gamma$ be a simple digraph of order $n$ and let $L$ be the Laplacian matrix of $\Gamma$.
Then, the following hold:
\begin{enumerate}[(i)]
\item For any restrictor matrix $Q$ of order $n$, the restricted numerical range satisfies $W_{r}(L)=W(Q^{*}LQ)$.
\item The set $W_{r}(L)$ is invariant under re-ordering of the vertices of $\Gamma$.
\item The eigenvalues of $L$ are contained in $W_{r}(L)$, except (possibly) for the zero eigenvalue associated with the eigenvector $\mathbf{e}$.
\item The minimum real part of $W_{r}(L)$ is equal to $\alpha(\Gamma)$ and the maximum real part of $W_{r}(L)$ is equal to $\beta(\Gamma)$.
\item Let $\comp{L}$ denote the Laplacian matrix of $\comp{\Gamma}$. Then, $W_{r}(\comp{L}) = n - W_{r}(L)$.
\end{enumerate}
\end{proposition}

The theorem below summarizes the known characterizations of digraphs using the restricted numerical range~\cite{Cameron2021_RNR}.
Note that the \emph{directed join} of the digraphs $\Gamma_{1} = (V_{1},E_{1})$ and $\Gamma_{2}=(V_{2},E_{2})$, where $V_{1}\cap V_{2}=\emptyset$, is defined by
\[
\Gamma_{1}\djoin\Gamma_{2} = \left(V_{1}\sqcup V_{2},E_{1}\sqcup E_{2}\sqcup\left\{(i,j)\colon~i\in V_{1},\ j\in V_{2}\right\}\right).
\]
Let $K_{n}$ and $\comp{K_{n}}$ denote the complete and empty digraphs of order $n$, respectively.
Then, the \emph{$k$-imploding} star is defined by $\Gamma=\comp{K_{n-k}}\djoin K_{k}$, for some $k\in\{0,1,\ldots,n\}$.
Also, a \emph{regular tournament} is defined as a tournament digraph of order $n$, where $n\geq 3$ is odd and the in-degree and out-degree of each vertex is equal to $(n-1)/2$.
Finally, $\Gamma = (V,E)$ is said to be \emph{bidirectional} if for every $(i,j)\in E$, we have $(j,i)\in E$.
\begin{theorem}\label{thm:rnr_char}
Let $\Gamma$ be a digraph of order $n$  and let $L$ be the Laplacian matrix of $\Gamma$.
Then, the following characterizations hold:
\begin{enumerate}[(i)]
\item $\Gamma$ is a dicycle (directed cycle) of order $n$ if and only if $W_{r}(L)$ is a complex polygon with vertices
\[
\left\{1-e^{\iu 2\pi j/n}\colon j=1,\ldots,n-1\right\},
\]
where $\iu$ denotes the imaginary unit.
\item $\Gamma$ is a $k$-imploding star if and only if $W_{r}(L)$ is a single point. Moreover, the numerical value of this point is $k$.
\item $\Gamma$ is a regular tournament if and only if $W_{r}(L)$ is a vertical line segment. Moreover, this vertical line segment has real part $n/2$.
\item $\Gamma$ is a directed join of two bidirectional digraphs, where one may be the null digraph, if and only if $W_{r}(L)$ is a horizontal line segment. Moreover, this line segment lies on the non-negative portion of the real axis. 
\end{enumerate}
\end{theorem}

\begin{figure}[ht]
\centering
\includegraphics[width=1.0\textwidth]{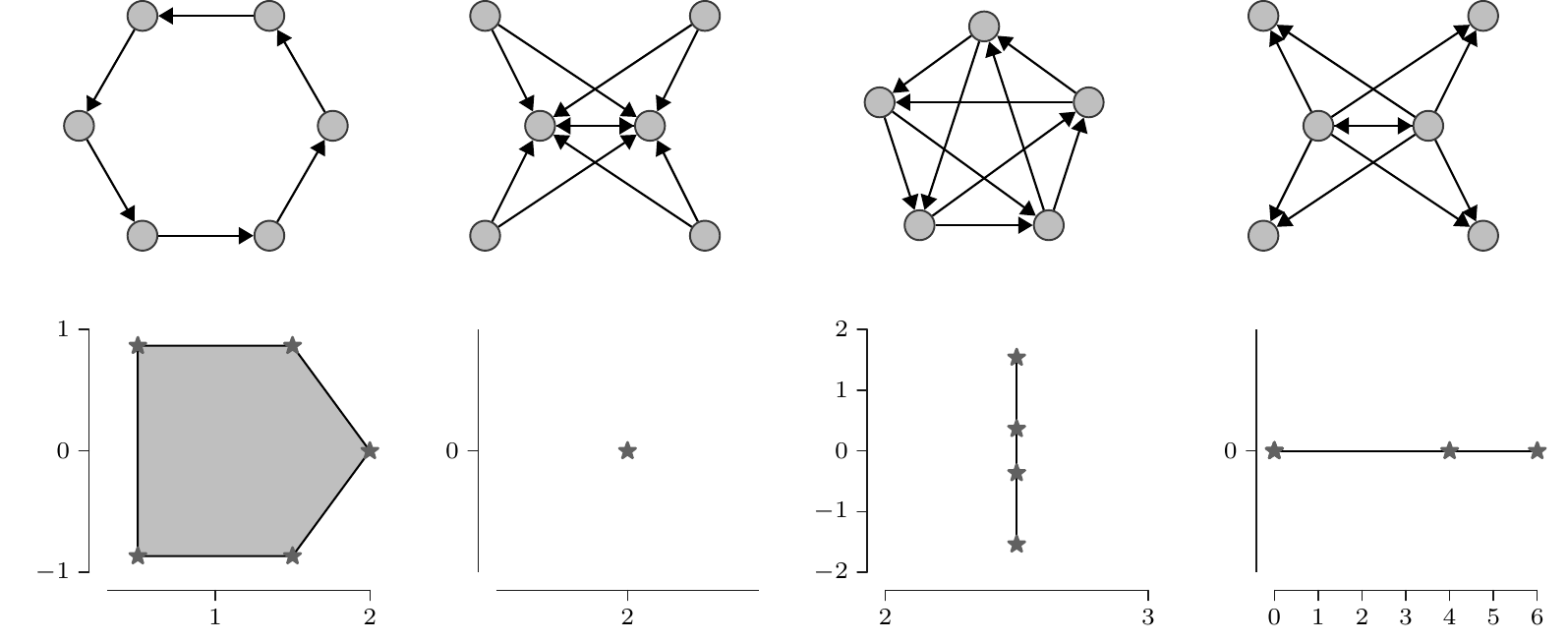}
\caption{Examples digraphs from each part of Theorem~\ref{thm:rnr_char}.}
\label{fig:rnr_char_ex}
\end{figure}
In Figure~\ref{fig:rnr_char_ex}, examples from each part of Theorem~\ref{thm:rnr_char} are shown.
Note that the digraph is shown above with the restricted numerical range below, and the eigenvalues of $Q^{*}LQ$ are displayed using the star symbol.
Each example can be viewed more generally as a \emph{polygonal digraph}, which is defined as a digraph whose restricted numerical range is equal to a convex polygon in the complex plane, that is, the convex hull of the eigenvalues of $Q^{*}LQ$~\cite{Cameron2022_RNR}.
In fact, parts (ii)--(iv) of Theorem~\ref{thm:rnr_char} characterize all degenerate polygonal digraphs, that is, digraphs with a restricted numerical range equal to a point or line segment. 
In general, polygonal digraphs can be split into three classes:
\emph{Normal digraphs} are those that have a normal Laplacian matrix, and therefore a normal restricted Laplacian matrix. \emph{Restricted-normal digraphs} are those whose Laplacian matrix is not normal but whose restricted Laplacian matrix is normal.
\emph{Pseudo-normal} digraphs are those whose Laplacian matrix and restricted Laplacian matrix are not normal but whose restricted numerical range is polygonal. 
Figure~\ref{fig:classes} provides an example from each class.
\begin{figure}[ht]
\centering
\includegraphics[width=0.75\textwidth]{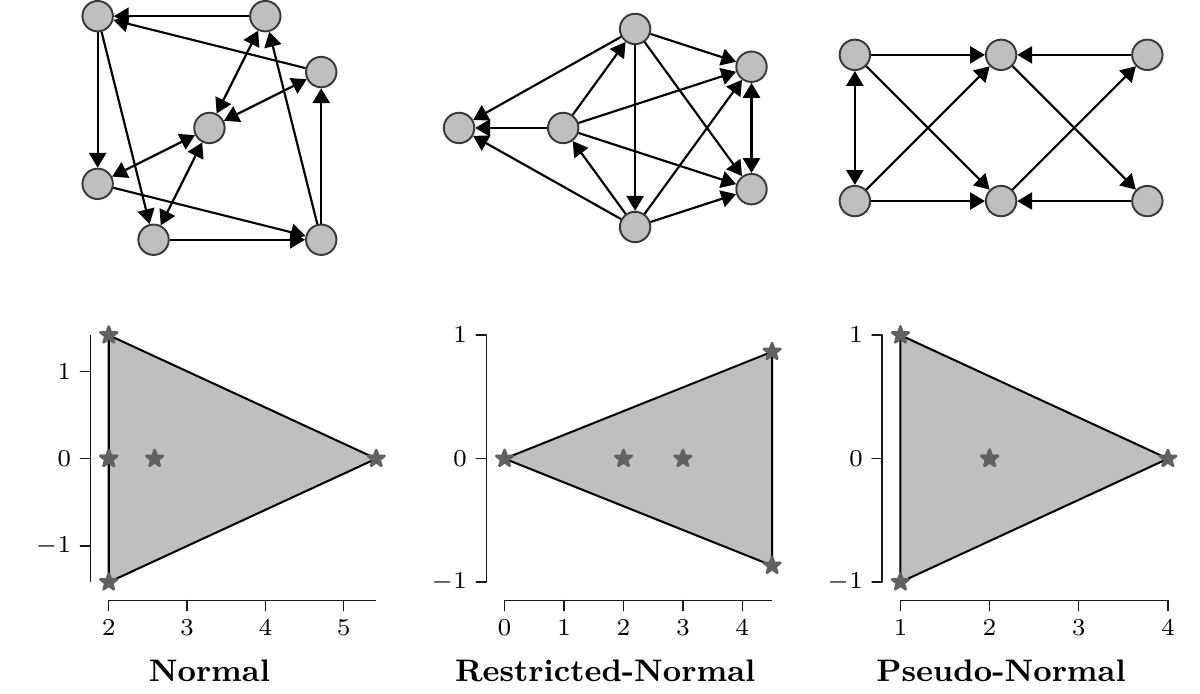}
\caption{Examples from each class of polygonal digraphs.}
\label{fig:classes}
\end{figure}

The dicycle and regular tournament shown in Figure~\ref{fig:rnr_char_ex} are also normal digraphs, since any digraph whose Laplacian matrix, possibly after re-ordering of vertices, can be written as a circulant matrix will be normal. 
Also, note that the restricted-normal digraph in Figure~\ref{fig:classes} can be viewed as the directed join of two normal digraphs, that is, an order $3$ dicycle and the disjoint union of $K_{2}$ and $K_{1}$. 
It turns out that the directed join of two normal digraphs is always a restricted-normal digraph~\cite[Proposition 3.6]{Cameron2022_RNR}; therefore, the $2$-imploding star and directed join of $K_{2}$ and $\comp{K_{4}}$ in Figure~\ref{fig:rnr_char_ex} are also restricted normal digraphs.
Finally, note that by~\cite[Theorem 3]{Johnson1976}, pseudo-normal digraphs have a restricted-normal Laplacian matrix that is unitarily similar to $D\oplus C$, where $D$ is a diagonal matrix, $C$ is not normal, and $W(C)\subset W(D)$.
\section{The Laplacian Spread of Digraphs}\label{sec:spread}
We define the \emph{Laplacian spread} of a digraph $\Gamma$ by
\[
\sp{\Gamma} = \beta(\Gamma) - \alpha(\Gamma),
\]
that is, $\sp{\Gamma}$ is equal to the length of the real part of $W_{r}(L)$. 
The following results are direct corollaries of Theorem~\ref{thm:rnr_char}.
\begin{corollary}\label{cor:cycle_spread}
If $\Gamma$ is a directed cycle of order $n$, then the Laplacian spread satisfies
\[
\sp{\Gamma} = 
\begin{cases}
1+\cos\left(\frac{2\pi}{n}\right) & \text{if $n$ is even,} \\
\cos\left(\frac{2\pi}{n}\right) - \cos\left(\pi\frac{n-1}{n}\right) & \text{if $n$ is odd.} \\
\end{cases}
\]
\end{corollary}
\begin{proof}
By Theorem~\ref{thm:rnr_char}(i), the restricted numerical range of a directed cycle of order $n$ satisfies
\[
W_{r}(L) = \left\{1-e^{\iu 2\pi j/n}\colon j=1,\ldots,n-1\right\}.
\]
Hence, the algebraic connectivity of $\Gamma$ satisfies
\[
\alpha(\Gamma) = \re{1-e^{\iu 2\pi/n}} = 1 - \cos\left(2\pi/n\right).
\]
Furthermore, we have
\begin{align*}
\beta(\Gamma) &= \re{1-e^{\iu 2\pi\floor{n/2}/n}} \\
&= \begin{cases} 2 & \text{if $n$ is even,} \\ 1 - \cos\left(\pi\frac{n-1}{n}\right) & \text{if $n$ is odd.}\end{cases}
\end{align*}
\end{proof}
\begin{corollary}\label{cor:zero_spread}
The Laplacian spread satisfies
\[
\sp{\Gamma}=0
\]
if and only if $\Gamma$ is a k-imploding star or a regular tournament. 
\end{corollary}
\begin{proof}
By Proposition~\ref{prop:basic-rnr}(i), it follows that all of the properties of the numerical range also hold for the restricted numerical range. In particular, by the Toeplitz-Hausdorff theorem~\cite{Hausdorff1919,Toeplitz1918}, $W_{r}(L)$ is convex. Furthermore, since $L$ has real entries, it follows that $W_{r}(L)$ is symmetric with respect to the real-axis.
Hence, $\alpha(\Gamma)=\beta(\Gamma)$ if and only if $W_{r}(L)$ is a single point or a vertical line segment. 
The result now follows from Theorem~\ref{thm:rnr_char}(ii)--(iii).
\end{proof}

Next, we show that the order of a polygonal digraph provides a sharp bound on its Laplacian spread.
First, we prove the following lemma, which will also be used in Section~\ref{sec:balanced}. 
Note that $\sigma(A)$ denotes the multiset of all eigenvalues of the matrix $A$. 
\begin{lemma}\label{lem:restricted-spectra}
Let $A$ be a $n\times n$ matrix with eigenvector $\mathbf{e}$ corresponding to the eigenvalue $0$. 
Then, for any restrictor matrix $Q$ of order $n$, the spectra satisfy
\[
\sigma(A) = \sigma(Q^{*}AQ)\cup\{0\}.
\]
\end{lemma}
\begin{proof}
Let $\hat{Q}$ denote an $n\times  n$ unitary matrix whose first $(n-1)$ columns are equal to the columns of $Q$ and whose $n$th column is equal to the normalized all ones vector.
Then, since $A\mathbf{e}=0$, we have
\[
\hat{Q}^{*}A\hat{Q} = \begin{bmatrix}Q^{*}AQ & 0 \\ \frac{1}{\sqrt{n}}\mathbf{e}^{T}AQ & 0 \end{bmatrix}.
\]
\end{proof}
Note that the proof of Theorem~\ref{thm:polygonal_spread} will utilize the following notation: $I_{n}$ denotes the $n\times n$ identity matrix and $J_{m\times n}$ denotes the $m\times n$ all ones matrix. 
\begin{theorem}\label{thm:polygonal_spread}
If $\Gamma$ is a polygonal digraph of order $n$, then the Laplacian spread satisfies
\[
\sp{\Gamma} \leq n.
\]
Moreover, this bound is sharp for all polygonal digraphs of order $n\geq 4$.
\end{theorem}
\begin{proof}
Let $\Gamma$ be a polygonal digraph of order $n$.
Then, $W_{r}(\Gamma)$ is equal to the convex hull of the eigenvalues of $Q^{*}LQ$, where $Q$ is any restrictor matrix of order $n$. 
By Lemma~\ref{lem:restricted-spectra}, it follows that $W_{r}(\Gamma)$ is equal to the convex hull of the eigenvalues of $L$, not including the zero eigenvalue of $L$ corresponding to the all ones eigenvector. 
Also, since every Laplacian matrix is an M-matrix, we know that the eigenvalues of $L$ have non-negative real part (see, for example, Chapter 6, Theorem 4.6 in~\cite{Berman1994}); hence, $\alpha(\Gamma)\geq 0$.
Moreover, by Proposition~\ref{prop:basic-rnr}(v), $\Gamma$ is polygonal if and only if $\comp{\Gamma}$ is polygonal, and $\beta(\Gamma) = n - \alpha(\comp{\Gamma})$.
Therefore, $\alpha(\Gamma) \geq 0$ and $\beta(\Gamma)\leq n$ for all polygonal digraphs $\Gamma$, which implies that $\sp{\Gamma}\leq n$.

Now, let $n\geq 4$, and define the restricted normal digraph
\[
\Gamma = K_{2}\djoin \comp{K_{n-2}}.
\]
The Laplacian matrix of $\Gamma$, possibly after re-ordering the vertices, can be written as
\[
L(\Gamma) = \begin{bmatrix}L(K_{2}) + (n-2)I_{2} & -J_{2\times(n-2)} \\ 0 & L(\comp{K_{n-2}}) \end{bmatrix}.
\]
Since $L(\Gamma)$ has rank $2$, it follows that its zero eigenvalue has geometric multiplicity $(n-2)$; hence, $\alpha(\Gamma)=0$ and $\beta(\Gamma)=n$, which implies that $\sp{\Gamma}=n$.
\end{proof}

It is worth noting that not all digraphs have a Laplacian spread bounded above by their order.
For instance, the digraph in Figure~\ref{fig:large-spread} has order $5$ and spread $5.035$, rounded to three decimal places.
By~\cite[Lemma 9]{Wu2005}, the Laplacian spread of any digraph $\Gamma$, with non-negative weights, satisfies
\begin{equation}\label{eq:wght-spread-bound}
\sp{\Gamma} \leq \frac{1}{2}\max_{v\in V}\left(3d^{+}(v)+d^{-}(v)\right) - \frac{1}{2}\min_{v\in V}\left(d^{+}(v)-d^{-}(v)\right).
\end{equation}
\begin{figure}[ht]
\centering
\includegraphics[width=0.50\textwidth]{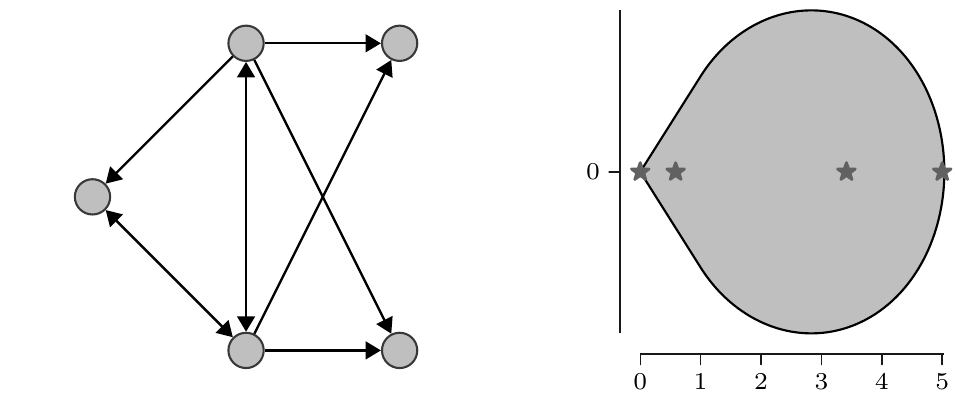}
\caption{A digraph with Laplacian spread larger than its order.}
\label{fig:large-spread}
\end{figure}
The bound in~\eqref{eq:wght-spread-bound} is equal to $8$ for the digraph in Figure~\ref{fig:large-spread}.
In general, for any unweighted digraph, the bound in~\eqref{eq:wght-spread-bound} has the following maximum value
\begin{equation}\label{eq:unwght-spread-bound}
\sp{\Gamma} \leq \frac{5}{2}(n-1). 
\end{equation}
We suspect that this bound is not sharp, and can be reduced to a linear function in $n$ with a slope of $1$. 

On a separate note, the digraph that attains the upper bound in Theorem~\ref{thm:polygonal_spread} is restricted-normal since it is defined as the directed join of two normal digraphs.
However, since normal digraphs are also balanced, Corollary~\ref{cor:balanced_spread} implies that the bound in Theorem~\ref{thm:polygonal_spread} is not sharp for normal digraphs. 
Yet, by Theorem~\ref{thm:pseudo-normal-spread}, this bound is sharp for pseudo-normal digraphs of order $n\geq 7$. 
Note that the proof of this result will utilize the following notation: $\mathbf{e}^{n}$ denotes the all ones vector of dimension $n$. 
Moreover, we let $T_{n}$ denote the regular tournament of order $n$ whose adjacency matrix has eigenvalues with maximum and minimum imaginary part equal to
\[
\pm\frac{1}{2}\cot\left(\frac{\pi}{2n}\right),
\]
respectively, which is the largest possible imaginary part of an eigenvalue associated with the adjacency matrix of a tournament digraph of order $n$ according to Pick's inequality~\cite{Pick1922}.
Such tournaments are known as \emph{regular Pick tournaments} and are unique up to graph isomorphism (see, for example Corollary 3.2 and Example 3.1 in~\cite{Gregory1993}).
Moreover, their adjacency matrix, possibly after reordering the vertices, can be written as
\[
A = P + P^{2} + \cdots + P^{\frac{n-1}{2}},
\]
where $n\geq 3$ is odd and $P$ is the permutation matrix corresponding to the permutation $(2,3,\ldots,n,1)$. 
Since the Laplacian $L=\frac{n-1}{2}I - A$ is a circulant matrix, it follows that
\[
\sigma(L) = \left\{\frac{n-1}{2}-\sum_{s=1}^{\frac{n-1}{2}}e^{\iu 2\pi j s/n}\colon~j=0,1,\ldots,n-1\right\}.
\]
\begin{theorem}\label{thm:pseudo-normal-spread}
Let $\Gamma = \left(K_{2}\djoin T_{p}\right)\djoin \comp{K_{t}}$, where $p\geq 3$ is an odd integer and $t$ is an integer between $2$ and $p$.
Then, $\Gamma$ is a pseudo-normal digraph with spread $\sp{\Gamma}=n$, where $n=2+p+t$ is the order of $\Gamma$.
\end{theorem}
\begin{proof}
Note that the Laplacian matrix of $\Gamma$, possibly after re-ordering the vertices, can be written as
\[
L(\Gamma) = \begin{bmatrix}L(K_{2}) + (p+t)I_{2} & -J_{2\times p} & -J_{2\times t} \\ 0 & L(T_{p}) + tI_{p} & -J_{p\times t} \\ 0 & 0 & L(\comp{K_{t}}) \end{bmatrix}.
\]
Since $K_{2}$ is a normal digraph, $L(K_{2})$ has orthonormal eigenvectors
\[
\mathbf{u}_{1},\frac{1}{\sqrt{2}}\mathbf{e}^{2}
\]
corresponding to the eigenvalues $\lambda_{1},\lambda_{2}$.
Similarly, $L(T_{p})$ has orthonormal eigenvectors
\[
\mathbf{v}_{1},\ldots,\mathbf{v}_{p-1},\frac{1}{\sqrt{p}}\mathbf{e}^{p}
\]
corresponding to the eigenvalues $\mu_{1},\ldots,\mu_{p}$, and $L(\comp{K_{t}})$ has orthonormal eigenvectors
\[
\mathbf{w}_{1},\ldots,\mathbf{w}_{t-1},\frac{1}{\sqrt{t}}\mathbf{e}^{t}
\]
corresponding to the eigenvalues $\nu_{1},\ldots,\nu_{t}$.

Now, we construct a restrictor matrix $Q=[\mathbf{q}_{1},\ldots,\mathbf{q}_{n-1}]$ of order $n$, where
\[
\mathbf{q}_{1}=\begin{bmatrix}\mathbf{u}_{1} \\ 0 \\ 0 \end{bmatrix},
\mathbf{q}_{2} = \begin{bmatrix}0 \\ \mathbf{v}_{1} \\ 0\end{bmatrix},
\ldots,
\mathbf{q}_{p} = \begin{bmatrix}0 \\ \mathbf{v}_{p-1} \\ 0\end{bmatrix},
\mathbf{q}_{p+1} = \begin{bmatrix}0 \\ 0 \\ \mathbf{w}_{1} \end{bmatrix},
\ldots,
\mathbf{q}_{p+t-1} = \begin{bmatrix}0 \\ 0 \\ \mathbf{w}_{t-1} \end{bmatrix},
\]
and
\[
\mathbf{q}_{p+t} = \frac{1}{\sqrt{2p^{2}+4p}}\begin{bmatrix} -p\mathbf{e}^{2} \\ 2\mathbf{e}^{p} \\ 0 \end{bmatrix},
\mathbf{q}_{p+t+1} = \frac{1}{\sqrt{2+p+\frac{1}{t}(4+4p+p^{2})}}\begin{bmatrix} \mathbf{e}^{2} \\ \mathbf{e}^{p} \\ -\frac{2+p}{t}\mathbf{e}^{t}\end{bmatrix}.
\]
Then, the restricted Laplacian can be written as
\[
Q^{*}L(\Gamma)Q = D\oplus T,
\]
which, by~\cite[Proposition 1.2.10]{Horn1991}, implies that the restricted numerical range satisfies
\[
W_{r}(\Gamma)=\conv{W(D)\cup W(T)}.
\]
Note that $D$ is a diagonal matrix with diagonal entries
\begin{equation}\label{eq:pns-eigs}
\lambda_{1}+p+t,\mu_{1}+t,\ldots,\mu_{p-1}+t,\nu_{1},\ldots,\nu_{t-1}
\end{equation}
and $T$ is a lower triangular matrix of the form
\[
T = \begin{bmatrix}p+t & 0 \\ -\sqrt{\frac{2pt}{2+p+t}} & t \end{bmatrix}.
\]

Since $D$ is diagonal, its numerical range is equal to the convex hull of its eigenvalues, that is, the values in~\eqref{eq:pns-eigs}, where $\lambda_{1}=2$, $\nu_{1},\ldots,\nu_{t-1}=0$, and $\mu_{1},\ldots,\mu_{p-1}$ lie on a vertical line with real part equal to $p/2$ and with maximum and minimum imaginary part equal to $\pm\frac{1}{2}\cot\left(\frac{\pi}{2p}\right)$, respectively.
Therefore, $W(D)$ is a quadrilateral with vertices
\begin{equation}\label{eq:quad}
\left(0,0\right),~\left(\frac{p}{2}+t,\frac{1}{2}\cot\left(\frac{\pi}{2p}\right)\right),~\left(2+p+t,0\right),~\left(\frac{p}{2}+t,-\frac{1}{2}\cot\left(\frac{\pi}{2p}\right)\right).
\end{equation}
Furthermore, by the elliptical range theorem~\cite{Li1996}, the boundary of $W(T)$ is an ellipse with center $(\frac{p}{2}+t,0)$, foci $t$ and $p+t$, and minor and major axis length of
\begin{equation}\label{eq:elli}
\sqrt{\frac{2pt}{2+p+t}}~\text{and}~\sqrt{p^{2}+\frac{2pt}{2+p+t}},
\end{equation}
respectively. 

Using a computer algebra system, such as Mathematica, one can verify that the boundaries of $W(D)$ and $W(T)$ don't intersect for $p\geq 3$ and $2\leq t\leq p$. 
Hence, since the center of the ellipse $W(T)$ is contained in $W(D)$, it follows that $W(T)\subset W(D)$.
Therefore, $\Gamma$ is pseudo-normal and has Laplacian spread
\[
\sp{\Gamma} = (2+p+t) - 0 = n.
\]
\end{proof}

\begin{figure}[ht]
\centering
\includegraphics[width=0.75\textwidth]{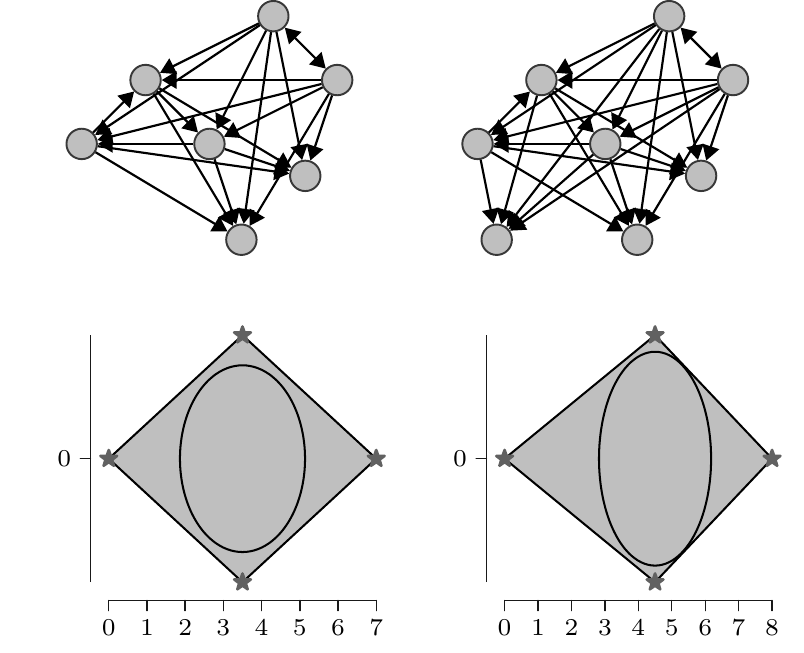}
\caption{Restricted numerical range for digraphs in Theorem~\ref{thm:pseudo-normal-spread} with $p=3$ and $2\leq t\leq 3$.}
\label{fig:pseudo-normal-spread}
\end{figure}

In Figure~\ref{fig:pseudo-normal-spread}, we illustrate Theorem~\ref{thm:pseudo-normal-spread} by displaying the digraphs $(K_{2}\djoin T_{3})\djoin \comp{K_{2}}$ (left) and $(K_{2}\djoin T_{3})\djoin \comp{K_{3}}$ (right) and their restricted numerical range.
Note that the quadrilateral with vertices given in~\eqref{eq:quad} and the ellipse with minor and major axis length given in~\eqref{eq:elli} are clearly displayed.
Moreover, as indicated in the proof of Theorem~\ref{thm:pseudo-normal-spread}, the ellipse is contained inside of the quadrilateral, which implies that both digraphs shown are pseudo-normal. 

\section{Balanced Digraphs}\label{sec:balanced}
In this section, we show that the Laplacian spread has a sharp bound of $(n-1)$ for all balanced digraphs of order $n$. 
To this end, note that Proposition~\ref{prop:basic-rnr}(v) implies that $\alpha(\Gamma)+\beta(\comp{\Gamma})=n$.
Hence, the statement 
\begin{equation}\label{eq:bal-lap-spread}
\sp{\Gamma} = \beta(\Gamma)-\alpha(\Gamma) \leq n-1
\end{equation}
is equivalent to the statement
\begin{equation}\label{eq:bal-lap-spread2}
\alpha(\Gamma) + \alpha(\comp{\Gamma}) \geq 1.
\end{equation}

Now, let $\Gamma=(V,E)$ be a \emph{balanced digraph}, that is, $d^{+}(i)=d^{-}(i)$ for all $i\in V$.
Also, let $L$ denote the Laplacian matrix of $\Gamma$.
Then, $\mathbf{e}$ is an eigenvector of $L$ and $L^{T}$ corresponding to the zero eigenvalue. 
Hence, $\mathbf{e}$ is an eigenvector of the Hermitian (symmetric) part of the Laplacian:
\[
H(L) = \frac{1}{2}\left(L+L^{T}\right),
\]
corresponding to the zero eigenvalue. 

By Proposition~\ref{prop:basic-rnr}(iv) and~\cite[Theorem 9]{Kippenhahn1951,Zachlin2008}, for any restrictor matrix $Q$ of order $n$, $\alpha(\Gamma)$ and $\beta(\Gamma)$ are equal to the minimum and maximum eigenvalues, respectively, of the Hermitian part of the restricted Laplacian:
\[
H(Q^{*}LQ) = \frac{1}{2}\left(Q^{*}LQ + Q^{*}L^{T}Q\right).
\]
Furthermore, since $H(L)\mathbf{e}=0$, Lemma~\ref{lem:restricted-spectra} implies that the spectrum of $H(Q^{*}LQ) = Q^{*}H(L)Q$ and the spectrum of $H(L)$ only differ by the multiplicity of the zero eigenvalue.
Since $H(L)$ is an M-matrix, its eigenvalues are non-negative and it follows that $\beta(\Gamma)$ is equal to the maximum eigenvalue of $H(L)$ and $\alpha(\Gamma)$ is equal to the second smallest eigenvalue of $H(L)$.
Moreover, since $\Gamma$ is balanced if and only if $\comp{\Gamma}$ is balanced, it follows that $\alpha(\comp{\Gamma})$ is equal to the second smallest eigenvalue of $H(\comp{L})$, where $\comp{L}$ denotes the Laplacian matrix of $\comp{\Gamma}$.

Now, we follow the development in~\cite{Afshari2019} to prove an equivalent statement to~\eqref{eq:bal-lap-spread2} for all balanced digraphs $\Gamma$.
We begin with the following lemma.
\begin{lemma}\label{lem:bal_quad_form}
Let $\Gamma$ be a balanced digraph of order $n\geq 2$ and let $\mathbf{x}\in\mathbb{R}^{n}$.
Then,
\[
\mathbf{x}^{T}L\mathbf{x} = \frac{1}{2}\sum_{(i,j)\in E}\left(x_{i}-x_{j}\right)^{2}.
\]
\end{lemma}
\begin{proof}
In general, the quadratic form of any Laplacian matrix can be written as
\[
\mathbf{x}^{T}L\mathbf{x} = \sum_{(u,v)\in E}x_{u}^{2}-x_{u}x_{v}.
\]
Since $\Gamma$ is balanced, it follows that the edge set is an edge-disjoint union of dicycles. 
In particular, $E = \bigsqcup_{i=1}^{l}C_{i}$, where $C_{i} = \left\{(u_{1},u_{2}),(u_{2},u_{3}),\ldots,(u_{n_{i}},u_{1})\right\}$ is the edge set corresponding to a dicycle of length $n_{i}$. 
The quadratic form over $C_{i}$ can be written as
\begin{align*}
\sum_{(u,v)\in C_{i}}x_{u}^{2}-x_{u}x_{v} &= (x_{u_{1}}^{2}-x_{u_{1}}x_{u_{2}}) + (x_{u_{2}}^{2}-x_{u_{2}}x_{u_{3}}) + \cdots + (x_{u_{n_{i}}}^{2}-x_{u_{n_{i}}}x_{u_{1}}) \\
&= \frac{1}{2}\sum_{(u,v)\in C_{i}}(x_{u}-x_{v})^{2}.
\end{align*}
Hence, the quadratic form of the Laplacian of a balanced digraph can be written as
\begin{align*}
\mathbf{x}^{T}L\mathbf{x} &= \sum_{(u,v)\in E}x_{u}^{2}-x_{u}x_{v} \\
&= \sum_{i=1}^{l}\frac{1}{2}\sum_{(u,v)\in C_{i}}(x_{u}-x_{v})^{2} \\
&= \frac{1}{2}\sum_{(u,v)\in E}(x_{u}-x_{v})^{2}.
\end{align*}
\end{proof}

We are now ready to prove the main result of this section.
\begin{theorem}\label{thm:bal_spread_equivalence}
The following statements are equivalent:
\begin{enumerate}[(i)]
\item For any balanced digraph of order $n\geq 2$,
\[
\alpha(\Gamma) + \alpha(\comp{\Gamma}) \geq 1.
\]
\item For any two orthonormal vectors $\mathbf{x},\mathbf{y}\in\mathbb{R}^{n}$ with zero mean and $n\geq 2$,
\[
\norm{\nabla_{\mathbf{x}}-\nabla_{\mathbf{y}}}^{2} \geq 2,
\]
where $\nabla_{\mathbf{x}}\in\mathbb{R}^{\binom{n}{2}}$ is the vector whose $ij$ entry is equal to $\abs{x_{i}-x_{j}}$, for all $i<j$. 
\end{enumerate}
\end{theorem}
\begin{proof}
For $(ii)\Rightarrow(i)$, let $\mathbf{x},\mathbf{y}$ be real unit eigenvectors of $H(L)$ and $H(\comp{L})$, respectively, corresponding to $\alpha(\Gamma)$ and $\alpha(\comp{\Gamma})$.
Since $H(L)\mathbf{e}=0$ and $H(\comp{L})\mathbf{e}=0$, it follows that $\mathbf{x},\mathbf{y}\in\mathbf{e}^{\perp}$. 
Furthermore, since $H(\comp{L}) = \left(nI - \mathbf{e}\mathbf{e}^{T}\right) - H(L)$,
it follows that $\mathbf{y}$ is an eigenvector of $H(L)$ corresponding to $\beta(\Gamma)$, which implies that $\mathbf{x}\perp\mathbf{y}$.
Hence,
\[
\mathbf{x}' = \frac{1}{\sqrt{2}}\left(\mathbf{x}+\mathbf{y}\right)
~\text{and}~
\mathbf{y}' = \frac{1}{\sqrt{2}}\left(\mathbf{x}-\mathbf{y}\right)
\]
are orthonormal vectors with zero mean.
Now,
\begin{align*}
\alpha(\Gamma) + \alpha(\comp{\Gamma}) &= \mathbf{x}^{T}H(L)\mathbf{x} + \mathbf{y}^{T}H(\comp{L})\mathbf{y} \\
&= \mathbf{x}^{T}L\mathbf{x} + \mathbf{y}^{T}\comp{L}\mathbf{y} \\
&= \frac{1}{2}\sum_{(i,j)\in E}\left(x_{i}-x_{j}\right)^{2} + \frac{1}{2}\sum_{(i,j)\notin E}\left(y_{i}-y_{j}\right)^{2},
\end{align*}
where the last line follows from Lemma~\ref{lem:bal_quad_form}.
Since $2\min\{a,b\}=a+b-\abs{a-b}$ for any $a,b\in\mathbb{R}$, we can write the above equation as
\begin{align*}
\alpha(\Gamma) + \alpha(\comp{\Gamma}) &\geq \frac{1}{2}\sum_{i,j=1}^{n}\min\left\{\left(x_{i}-x_{j}\right)^{2},\left(y_{i}-y_{j}\right)^{2}\right\} \\
&= \frac{1}{4}\sum_{i,j=1}^{n}\left(\left(x_{i}-x_{j}\right)^{2} + \left(y_{i}-y_{j}\right)^{2} - \abs{\left(x_{i}-x_{j}\right)^{2} - \left(y_{i}-y_{j}\right)^{2}}\right) \\
&=\frac{1}{4}\sum_{i,j=1}^{n}\frac{1}{2}\left(\abs{\left(x_{i}-x_{j}\right)+\left(y_{i}-y_{j}\right)} - \abs{\left(x_{i}-x_{j}\right)-\left(y_{i}-y_{j}\right)}\right)^{2} \\
&=\frac{1}{4}\sum_{i,j=1}^{n}\left(\abs{\frac{x_{i}+y_{i}}{\sqrt{2}} - \frac{x_{j}+y_{j}}{\sqrt{2}}} - \abs{\frac{x_{i}-y_{i}}{\sqrt{2}} - \frac{x_{j}-y_{j}}{\sqrt{2}}}\right)^{2} \\
&= \frac{1}{4}\sum_{i,j=1}^{n}\left(\abs{x_{i}' - x_{j}'} - \abs{y_{i}' - y_{j}'}\right)^{2} \\
&= \frac{1}{4}\sum_{i<j}\left(\abs{x_{i}' - x_{j}'} - \abs{y_{i}' - y_{j}'}\right)^{2} + \frac{1}{4}\sum_{i>j}\left(\abs{x_{j}' - x_{i}'} - \abs{y_{j}' - y_{i}'}\right)^{2} \\
&=\frac{1}{2}\norm{\nabla_{\mathbf{x}}-\nabla_{\mathbf{y}}}^{2} \geq 1.
\end{align*}

Conversely, for $(i)\Rightarrow(ii)$, let $\mathbf{x},\mathbf{y}\in\mathbb{R}^{n}$ be two orthonormal vectors with zero mean. 
Then, define $\Gamma$ to be a digraph with vertex set $\{1,2,\ldots,n\}$, where
\[
(i,j)\in E~\Leftrightarrow~\left(x_{i}-x_{j}\right)\left(y_{i}-y_{j}\right) < 0.
\]
Note that $\Gamma$ is a bidirectional digraph since $(i,j)\in E$ if and only if $(j,i)\in E$. 
Now,
\begin{align*}
\norm{\nabla_{\mathbf{x}}-\nabla_{\mathbf{y}}}^{2} &= \frac{1}{2}\sum_{i<j}\left(\abs{x_{i}-x_{j}} - \abs{y_{i}-y_{j}}\right)^{2} + \frac{1}{2}\sum_{i>j}\left(\abs{x_{j}-x_{i}} - \abs{y_{j}-y_{i}}\right)^{2}\\
&= \frac{1}{2}\sum_{(i,j)\in E}\left(\left(x_{i}-x_{j}\right)+\left(y_{i}-y_{j}\right)\right)^{2} + \frac{1}{2}\sum_{(i,j)\notin E}\left(\left(x_{i}-x_{j}\right)-\left(y_{i}-y_{j}\right)\right)^{2} \\
&= \left(\mathbf{x}+\mathbf{y}\right)^{T}H(L)\left(\mathbf{x}+\mathbf{y}\right) + \left(\mathbf{x}-\mathbf{y}\right)^{T}H(\comp{L})\left(\mathbf{x}-\mathbf{y}\right) \\
&\geq \alpha(\Gamma)\norm{\mathbf{x}+\mathbf{y}}^{2} + \alpha(\comp{\Gamma})\norm{\mathbf{x}-\mathbf{y}}^{2} \\
&= 2\left(\alpha(\Gamma)+\alpha(\comp{\Gamma})\right) \geq 2,
\end{align*}
where the second to last line follows from the Rayleigh quotient theorem~\cite[Theorem 4.2.2]{Horn2013}.
\end{proof}

In~\cite[Theorem 2]{Afshari2019}, the authors show that statement (ii) in Theorem~\ref{thm:bal_spread_equivalence} is equivalent to~\eqref{eq:lap-spread-graphs2} for all simple undirected and unweighted graphs $G$.
Since the latter statement was proven in~\cite[Theorem 1]{Einollahzadeh2021}, it follows that both statements in Theorem~\ref{thm:bal_spread_equivalence} hold true. 
Therefore, we have the following corollary.
\begin{corollary}\label{cor:balanced_spread}
If $\Gamma$ is a balanced digraph of order $n$, then the Laplacian spread satisfies
\[
\sp{\Gamma}\leq n-1.
\]
Moreover, this bound is sharp for all balanced digraphs of order $n\geq 2$.
\end{corollary}
\begin{proof}
The bound follows from Theorem~\ref{thm:bal_spread_equivalence} and~\cite[Theorem 1]{Einollahzadeh2021}.
Now, let $n\geq 2$ and define 
\[
\Gamma = K_{1}\sqcup K_{n-1}.
\]
Then, by~\cite[Theorem 3.1(i)]{Cameron2022_RNR}, the restricted numerical range of $\Gamma$ satisfies
\[
W_{r}(\Gamma) = \conv{W_{r}(K_{1})\cup W_{r}(K_{n-1})\cup\{0\}}
\]
Also, by~\cite[Theorem 3.2]{Cameron2021_RNR}, we know that $W_{r}(K_{1}) = \emptyset$ and $W_{r}(K_{n-1})=\{n-1\}$.
Hence, 
\[
W_{r}(\Gamma) = \conv{\{0,n-1\}},
\]
and it follows that $\sp{\Gamma}=n-1$.
\end{proof}
\section{Weighted Digraphs}\label{sec:wght-digraphs}
In this section, we investigate the Laplacian spread for weighted digraphs, that is, digraphs with weights between $0$ and $1$. 
To that end, define a simple weighted digraph by $\Gamma'=(V,E,w)$, where $w\colon V\times V\rightarrow [0,1]$ denotes the weight function and we use the convention that $w(i,j)=0$ if and only if $(i,j)\notin E$.
Also, define the complement digraph $\comp{\Gamma'}$ as the digraph with the same vertex set $V$ and weight function $\comp{w}\colon V\times V\rightarrow[0,1]$, where for all $(i,j)\in V\times V$ we have $\comp{w}(i,j)=1-w(i,j)$, if $i\neq j$, and $\comp{w}(i,j)=0$, if $i=j$.

It is important to note that throughout this section we will use $\Gamma'$ to denote a weighted digraph and $\Gamma$ to denote an unweighted digraph, that is, a digraph with weight function $w\colon V\times V\rightarrow\{0,1\}$.
Also, the Laplacian matrix, restricted numerical range, and algebraic connectivity are defined for weighted digraphs analogously to how they were defined for unweighted digraphs in Section~\ref{sec:rnr}. 
In fact, the basic properties stated in Proposition~\ref{prop:basic-rnr} still hold, though some of the characterizations stated in Theorem~\ref{thm:rnr_char} and the partial characterizations of polygonal digraphs in~\cite{Cameron2022_RNR} may no longer hold. 

Now, define $\mathcal{S}_{n}$ as the set of all weighted digraphs of order $n$.
Also, define the convex combination of $\Gamma'_{1}=\left(V,E_{1},w_{1}\right),\Gamma'_{2}=\left(V,E_{2},w_{2}\right)\in\mathcal{S}_{n}$ by
\[
\lambda_{1}\Gamma'_{1}+\lambda_{2}\Gamma'_{2} = \left(V,E_{1}\cup E_{2},\lambda_{1}w_{1}+\lambda_{2}w_{2}\right),
\]
where $\lambda_{1},\lambda_{2}\geq 0$ and $\lambda_{1}+\lambda_{2}=1$. 
Note that $\lambda_{1}w_{1}+\lambda_{2}w_{2}\colon V\times V\rightarrow[0,1]$,
which implies that $\mathcal{S}_{n}$ is a convex set.
In fact, $\mathcal{S}_{n}$ is a hypercube defined by the inequalities
\begin{equation}\label{eq:wght-digraph-polytope}
0 \leq w(i,j)\leq 1,
\end{equation}
for all $i,j\in V$ such that $i\neq j$. 
Note that the vertices of $\mathcal{S}_{n}$ are integral and correspond to the unweighted digraphs of order $n$. 
Therefore, we have the following result. 
\begin{theorem}\label{thm:digraph-conv-hull}
Every weighted digraph can be written as the convex combination of unweighted digraphs. 
\end{theorem}

We note that the result in Theorem~\ref{thm:digraph-conv-hull} appears elsewhere in the literature; for instance, the undirected case was alluded to in~\cite[Section 4.3]{Barrett2022}, where the authors argue that the Laplacian spread for undirected graphs, as stated in~\eqref{eq:lap-spread-graphs2}, also holds for weighted undirected graphs.
Moreover, we have an analogous result to that in Theorem~\ref{thm:digraph-conv-hull} for weighted balanced digraphs, which is stated and proven below. 
\begin{theorem}\label{thm:balanced-conv-hull}
Every weighted balanced digraph can be written as a convex combination of unweighted balanced digraphs. 
\end{theorem}
\begin{proof}
Let $\mathcal{B}_{n}$ denote the convex set of all weighted balanced digraphs of order $n$ and note that $\mathcal{B}_{n}$ can be viewed as a polytope obtained from the hypercube of $\mathcal{S}_{n}$ by adding the following constraints to~\eqref{eq:wght-digraph-polytope}:
\begin{equation}\label{eq:wght_balanced_polytope}
\sum_{\substack{j=1 \\ j\neq i}}^{n}w(i,j) = \sum_{\substack{j=1 \\ j\neq i}}^{n}w(j,i),
\end{equation}
for all $i$ in $V$.

Let $B=[b_{ij}]$ denote the \emph{arc-incidence matrix} of $K_{n}$, which is defined as a $n\times n(n-1)$ matrix where the rows and columns are indexed by the vertices and edges of $K_{n}$, respectively.
Moreover, $b_{ij}=-1$ if edge $e_{j}$ leaves vertex $v_{i}$, $b_{ij}=1$ if edge $e_{j}$ enters vertex $v_{i}$, and $b_{ij}=0$ otherwise. 
It is well-known that $B$ is \emph{totally unimodular}, that is, every minor of $B$ is equal to $0$, $+1$, or $-1$; for example, this result follows immediately from the sufficient conditions for total unimodularity in~\cite[Theorem 3]{Hoffman1956}.
Moreover, by~\cite[Theorem 2]{Hoffman1956}, the following inequalities
\[
B\mathbf{w}=0,~0\leq\mathbf{w}\leq 1,
\]
where $\mathbf{w}\in\mathbb{R}^{n(n-1)}$, describe a polytope with integral vertices. 
The result follows since this polytope describes the convex set $\mathcal{B}_{n}$ where the vertices correspond to the unweighted balanced digraphs.
\end{proof}
Next, we show that the algebraic connectivity is a concave function of weighted digraphs.
To that end, note that, by Proposition~\ref{prop:basic-rnr}(v), the algebraic connectivity of a weighted digraph $\Gamma'$ can be written as
\begin{align*}
\alpha(\Gamma') &= n - \beta(\comp{\Gamma'}) \\
&= n - \max_{\substack{\mathbf{x}\perp\mathbf{e} \\ \norm{\mathbf{x}}=1}}\left(\mathbf{x}^{T}L(\comp{\Gamma'})\mathbf{x}\right). \\
\end{align*}
Since the max of a quadratic form over a convex set is a convex function, it follows that the algebraic connectivity is a concave function of weighted digraphs.
This observation combined with Theorem~\ref{thm:balanced-conv-hull} implies the following result.
\begin{corollary}\label{cor:wght-balanced-spread}
The Laplacian spread satisfies $\sp{\Gamma'}\leq(n-1)$ for all weighted balanced digraphs $\Gamma'$. 
\end{corollary}
\begin{proof}
Let $\Gamma'$ denote a weighted balanced digraph and note that $\sp{\Gamma'}\leq(n-1)$ is equivalent to the statement $\alpha(\Gamma')+\alpha(\comp{\Gamma'})\geq 1$.
By Theorem~\ref{thm:balanced-conv-hull}, we have
\[
\Gamma' = \lambda_{1}\Gamma_{1}+\cdots+\lambda_{m}\Gamma_{m},
\]
where $\lambda_{1},\ldots,\lambda_{m}\geq 0$, $\lambda_{1}+\cdots+\lambda_{m}=1$, and $\Gamma_{1},\ldots,\Gamma_{m}$ are unweighted balanced digraphs.
By Corollary~\ref{cor:balanced_spread}, we have $\alpha(\Gamma_{j})+\alpha(\comp{\Gamma_{j}})\geq 1$ for all $j\in\{1,\ldots,m\}$, and the result follows since the algebraic connectivity is a concave function. 
\end{proof}

It is worth noting that since every weighted undirected graph can be viewed as a weighted digraph with bidirectional edges, Corollary~\ref{cor:wght-balanced-spread} also implies that the Laplacian spread satisfies $\sp{G'}\leq(n-1)$ for all weighted undirected graphs $G'$. 

On a related note, the bound in Theorem~\ref{thm:polygonal_spread} also holds for weighted polygonal digraphs, though we are unaware of a convexity argument for this result. 
Rather, if $\Gamma'$ is a weighted digraph, then both $L(\Gamma')$ and $L(\comp{\Gamma'})$ are M-matrices, and Proposition~\ref{prop:basic-rnr}(v) implies that $\Gamma'$ is polygonal if and only if $\comp{\Gamma'}$ is polygonal, and $\beta(\comp{\Gamma'}) = n - \alpha(\Gamma')$.
Hence, if $\Gamma'$ is polygonal, then $\alpha(\Gamma')\geq 0$ and $\beta(\Gamma')\leq n$, and it follows that $\sp{\Gamma'}\leq n$. 

\section{Summary and Open Conjectures}\label{sec:summary} 
In this article, we define the Laplacian spread of a digraph as the length of the real part of its restricted numerical range as defined in~\cite{Cameron2022_RNR,Cameron2021_RNR}.
The Laplacian spread values for several families of digraphs are shown in Corollaries~\ref{cor:cycle_spread} and~\ref{cor:zero_spread}, and a sharp upper bound on the Laplacian spread for all polygonal digraphs is proved in Theorems~\ref{thm:polygonal_spread} and~\ref{thm:pseudo-normal-spread}.
Moreover, in Corollary~\ref{cor:balanced_spread}, we prove that $\sp{\Gamma}\leq (n-1)$ is a sharp upper bound for all balanced digraphs of order $n\geq 2$.
In particular, Theorem~\ref{thm:bal_spread_equivalence} shows that the validity of this upper bound is equivalent to the statement in~\eqref{eq:lap-spread-graphs3}.
Since the latter statement is also equivalent to the statement in~\eqref{eq:lap-spread-graphs2}, which was proven in~\cite[Theorem 1]{Einollahzadeh2021}, the result in Corollary~\ref{cor:balanced_spread} follows. 
Finally, in Corollary~\ref{cor:wght-balanced-spread}, we prove that $\sp{\Gamma'}\leq(n-1)$ also holds for all balanced digraphs $\Gamma'$, with weights between $0$ and $1$. 
Specifically, Theorem~\ref{thm:balanced-conv-hull} shows that every weighted balanced digraph can be written as the convex combination of unweighted balanced digraphs.
Therefore, the concavity of the algebraic connectivity implies that the bound in Corollary~\ref{cor:balanced_spread} also holds for weighted balanced digraphs.

To conclude, we display $\left(\alpha(\Gamma),\alpha(\comp{\Gamma})\right)$ values for unweighted balanced and polygonal digraphs of order $6$ and $7$.
These values were computed in C++ using the Nauty~\cite{McKay2013} and Eigen~\cite{eigen} libraries.
The source code for these computations along with Python scripts for interacting with the data is available at~\url{https://github.com/trcameron/LaplacianSpreadDigraphs}.
We use this empirical data to form and state open conjectures regarding the families of digraphs that attain certain spread values.
Note that all conjectures are made for unweighted digraphs.
 
\begin{figure}[ht]
\centering
\begin{tabular}{cc}
\includegraphics[width=0.45\textwidth]{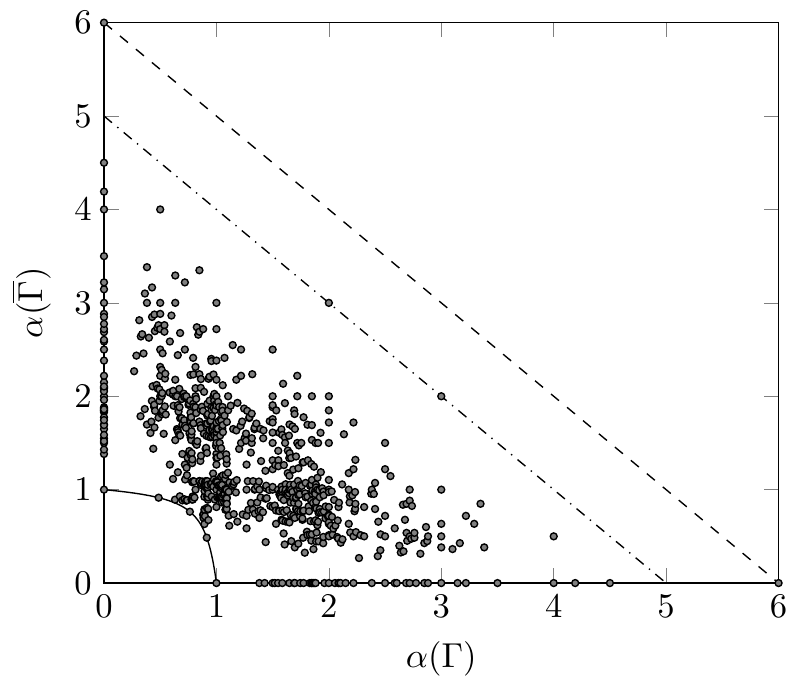}
&
\includegraphics[width=0.45\textwidth]{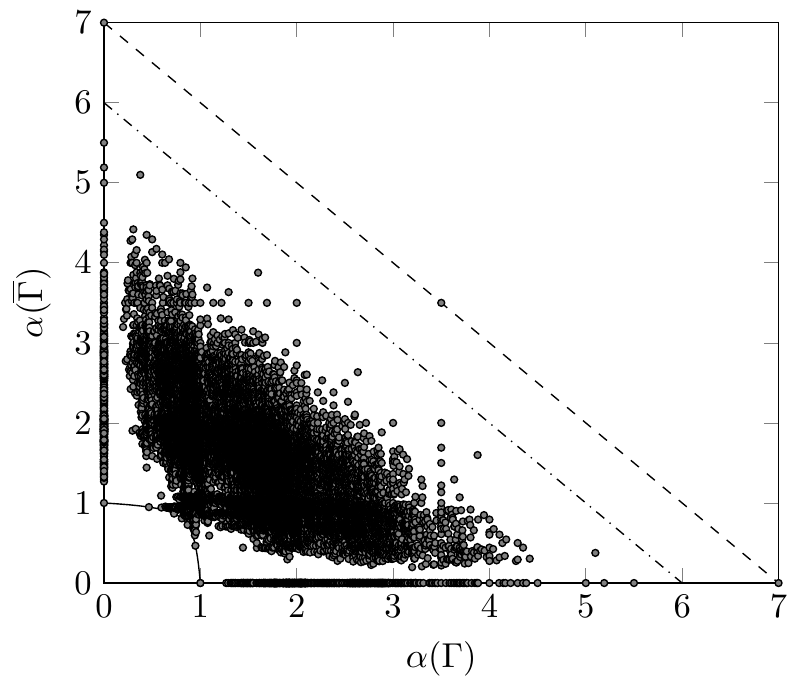}
\end{tabular}
\caption{$\left(\alpha(\Gamma),\alpha(\comp{\Gamma})\right)$ values for balanced digraphs of order $6$ (left) and $7$ (right).}
\label{fig:balanced_vals}
\end{figure}
Figure~\ref{fig:balanced_vals} displays $\left(\alpha(\Gamma),\alpha(\comp{\Gamma})\right)$ values for balanced digraphs of order $6$ (left) and $7$ (right). 
Note that the dashed line corresponds to the equation $y=n-x$; digraphs on this line have a Laplacian spread equal to zero.
Also, the dash-dotted line corresponds to the equation $y=n-1-x$; digraphs on this line have a Laplacian spread equal to one. 
Finally, the  solid curve corresponds to the implicit equation 
\[
xy(2-xy)=n(1-x)(1-y)(n-2-x-y),~0\leq x\leq 1,~0\leq y\leq 1,
\]
which algebraically defines the conjectured sharper bound on the Laplacian spread for unweighted graphs given in~\cite{Barrett2022}.

Corollary~\ref{cor:zero_spread} implies that the only balanced digraphs on the line $y=n-x$ are regular tournaments; hence, there is no balanced digraph on that line when $n$ is even. 
Moreover, we have the following conjectures.
\begin{conjecture}\label{con:bal_sp_gap}
There is no balanced digraph that satisfies $0<\sp{\Gamma}<1$. 
\end{conjecture}
\begin{conjecture}\label{con:bal_sp_leq1}
If $\Gamma$ is a balanced digraph with order $n\geq 3$ and $\sp{\Gamma}\in\{0,1\}$, then $\Gamma$ is a regular digraph. 
In particular, if $\sp{\Gamma}=0$ then $\Gamma$ is a regular tournament and if $\sp{\Gamma}=1$ then $n\geq 4$ is even and $\Gamma$ or $\comp{\Gamma}$ is a $\frac{n-2}{2}$-regular digraph.
\end{conjecture}
\begin{conjecture}\label{con:bal_sp_bound}
Let $\Gamma$ be a balanced digraph of order $n\geq 2$ and let $x=\alpha(\Gamma)$ and $y=\alpha(\comp{\Gamma})$.
If $x\leq 1$ and $y\leq 1$, then
\[
xy(2-xy)\geq n(1-x)(1-y)(n-2-x-y).
\]
\end{conjecture}

\begin{figure}[ht]
\centering
\begin{tabular}{cc}
\includegraphics[width=0.45\textwidth]{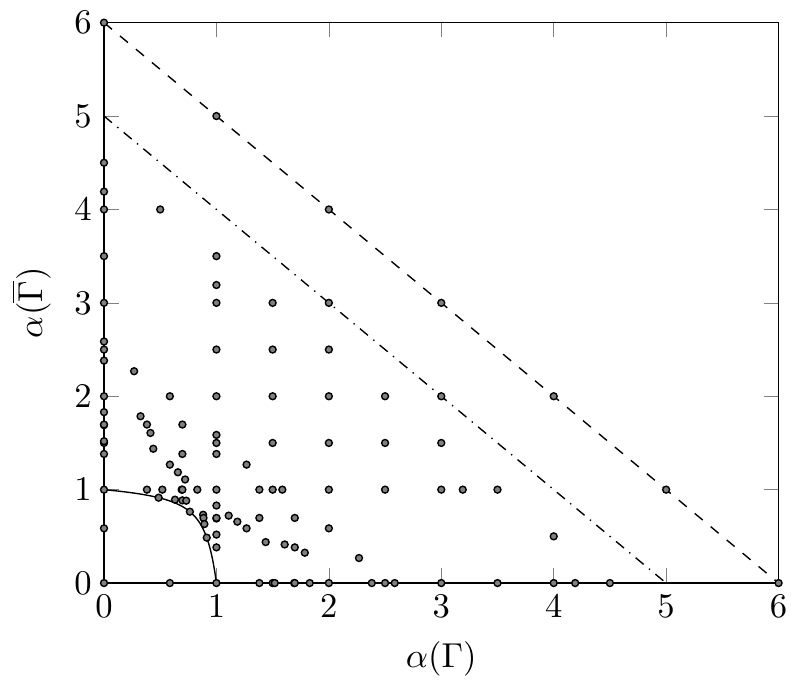}
&
\includegraphics[width=0.45\textwidth]{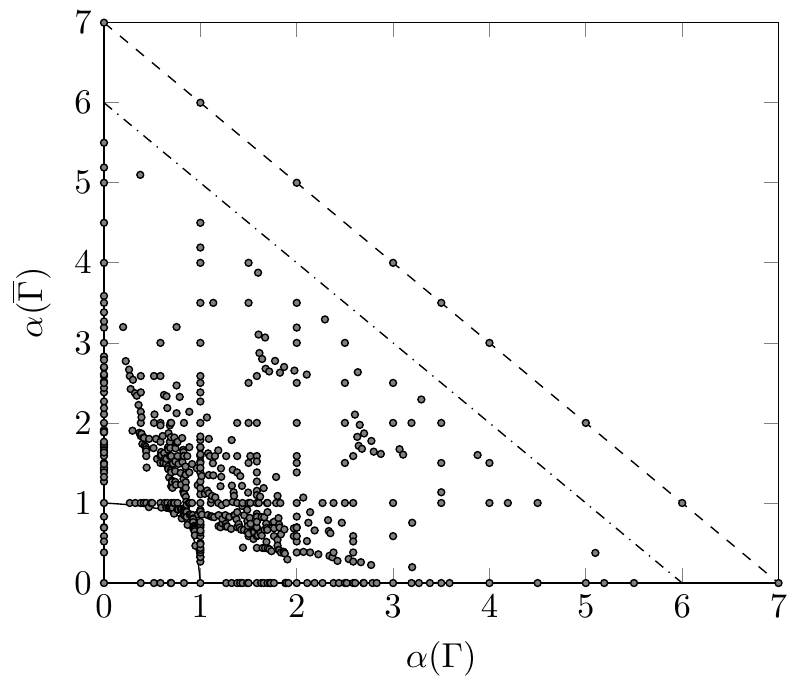}
\end{tabular}
\caption{$\left(\alpha(\Gamma),\alpha(\comp{\Gamma})\right)$ values for polygonal digraphs of order $6$ (left) and $7$ (right).}
\label{fig:polygonal_vals}
\end{figure}
Figure~\ref{fig:polygonal_vals} displays $\left(\alpha(\Gamma),\alpha(\comp{\Gamma})\right)$ values for polygonal digraphs of order $6$ (left) and $7$ (right). 
Note that the dashed and dash-dotted lines, as well as the solid curve, are the same as they were in Figure~\ref{fig:balanced_vals}. 

Corollary~\ref{cor:zero_spread} implies that the only digraphs on the line $y=n-x$ are regular tournaments or $k$-imploding stars.
Also, Theorems~\ref{thm:polygonal_spread} and~\ref{thm:pseudo-normal-spread} identify families of digraphs that lie at the origin, thus having a Laplacian spread equal to $n$.
Moreover, we have the following conjectures.
\begin{conjecture}\label{con:poly_sp_gap}
There is no polygonal digraph that satisfies $0<\sp{\Gamma}<1$. 
\end{conjecture}
\begin{conjecture}\label{con:poly_sp_leq1}
If $\Gamma$ be a polygonal digraph with order $n\geq 2$ and $\sp{\Gamma}\in\{0,1\}$, then $\Gamma$ is a regular digraph or a $k$-imploding star. 
In particular, if $\sp{\Gamma}=0$ then $\Gamma$ is a regular tournament or a $k$-imploding star and if $\sp{\Gamma}=1$ then $n\geq 4$ is even and $\Gamma$ or $\comp{\Gamma}$ is a $\frac{n-2}{2}$-regular digraph.
\end{conjecture}
\begin{conjecture}\label{con:poly_sp_bound}
Let $\Gamma$ be a balanced digraph of order $n\geq 2$ and let $x=\alpha(\Gamma)$ and $y=\alpha(\comp{\Gamma})$.
If $x,y\leq 1$ and
\[
xy(2-xy) < n(1-x)(1-y)(n-2-x-y),
\]
then $x=0$ or $y=0$.
\end{conjecture}

Finally, we formally state the conjecture made in Section~\ref{sec:spread} regarding the Laplacian spread of digraphs.
Note that any bound that is proven for the Laplacian spread of digraphs will hold more generally for weighted digraphs by Theorem~\ref{thm:digraph-conv-hull} and the concavity of the algebraic connectivity.
\begin{conjecture}\label{con:unwght-spread-bound}
The bound on the Laplacian spread of digraphs given in~\eqref{eq:unwght-spread-bound} can be reduced to a linear function in $n$ with a slope of $1$. 
\end{conjecture}

\section*{Acknowledgments}
The authors are indebted to Dr. Jonad Pulaj for noting the total unimodularity property of the arc-incidence matrix of a digraph used in the proof of Theorem~\ref{thm:balanced-conv-hull}.

\bibliographystyle{siam}
\bibliography{references}
\end{document}